\documentclass[12pt,reqno]{amsart}

\newcommand\version{August 28, 2019}

\usepackage{amsmath,amsfonts,amsthm,amssymb,amsxtra}
\newtheorem{theorem}{Theorem}%[section]
\newtheorem{proposition}[theorem]{Proposition}
\newtheorem{lemma}[theorem]{Lemma}
\newtheorem{corollary}[theorem]{Corollary}

\theoremstyle{definition}

\theoremstyle{remark}

\renewcommand{\epsilon}{\varepsilon}

\newcommand{\N}{\mathbb{N}}

\renewcommand{\phi}{\varphi}
\newcommand{\R}{\mathbb{R}}

\newcommand{\Sph}{\mathbb{S}}

\DeclareMathOperator{\spa}{span}

\begin{document}

\title[Non-degeneracy for the critical Lane--Emden system--\version]{Non-degeneracy for the\\ critical Lane--Emden system}

\author{Rupert L. Frank}
\address[R.~L.~Frank]{Mathematisches Institut, Ludwig-Maximilians-Universit\"at M\"unchen, Theresienstr. 39, 80333 M\"unchen, Germany, and Mathematics 253-37, Caltech, Pa\-sa\-dena, CA 91125, USA}
\email{r.frank@lmu.de, rlfrank@caltech.edu}

\author{Seunghyeok Kim}
\address[S.~Kim]{Department of Mathematics and Research Institute for Natural Sciences, College of Natural Sciences, Hanyang University, 222 Wangsimni-ro Seongdong-gu, Seoul 04763, Republic of Korea}
\email{shkim0401@hanyang.ac.kr, shkim0401@gmail.com}

\author{Angela Pistoia}
\address[A.~Pistoia]{Dipartimento SBAI,  ``Sapienza" Universit\`a di Roma, via Antonio Scarpa 16, 00161 Roma, Italy}
\email{angela.pistoia@uniroma1.it}

\begin{abstract}
We prove the non-degeneracy for the critical Lane--Emden system
$$ -\Delta U = V^p,\quad
-\Delta V = U^q,\quad
U, V > 0 \quad \text{in } \R^N
$$
for all $N \ge 3$ and $p,q > 0$ such that $\frac{1}{p+1} + \frac{1}{q+1} = \frac{N-2}{N}$.
We show that all solutions to the linearized system around a ground state must arise from the symmetries of the critical Lane--Emden system
provided that they belong to the corresponding energy space or they decay to 0 uniformly as the point tends to infinity.
\end{abstract}

\date{\today}
\thanks{R.~Frank was partially supported by US National Science Foundation grant DMS-1363432. S.~Kim was partially supported by Basic Science Research Program through the National Research Foundation of Korea (NRF) funded by the Ministry of Education (NRF2017R1C1B5076384). A. Pistoia was partially supported by Fondi di Ateneo ``Sapienza" Universit\`a di Roma (Italy).}
\maketitle
%\subjclass[2000]{Primary 35P15, Secondary 35J10, 47F05}

\maketitle

\renewcommand{\thefootnote}{${}$} \footnotetext{\copyright\, 2019 by
  the authors. This paper may be reproduced, in its entirety, for
  non-commercial purposes.}

\section{Introduction}
We consider the critical Lane--Emden system
\begin{equation}\label{eq-LEs}
\begin{cases}
-\Delta U = V^p &\text{in } \mathbb R^N,\\
-\Delta V = U^q &\text{in } \mathbb R^N,\\
U,V >0\ &\text{in } \mathbb R^N
\end{cases}
\end{equation}
where $N \ge 3$, $p,q>0$ and $(p,q)$ belongs to the {\em critical hyperbola}
\begin{equation}\label{eq-hyp}
\frac{1}{p+1} + \frac{1}{q+1} = \frac{N-2}{N}.
\end{equation}

In \cite[Corollary I.2]{Li}, Lions found a positive ground state
\[(U,V) \in \dot W^{2,\frac{p+1}{p}}(\R^N) \times\dot W^{2,\frac{q+1}{q}}(\R^N) \]
of \eqref{eq-LEs}, by transforming it into an equivalent scalar equation
\begin{equation}\label{eq-LEsc}
\Delta \left(|\Delta U|^{\frac{1}{p}-1} \Delta U\right) = |U|^{q-1}U \quad \text{in } \R^N
\end{equation}
and employing a concentration-compactness argument to the associated minimization problem
\begin{equation}\label{eq-K_pq}
\begin{aligned}
\inf \left\{\|\Delta u\|_{L^\frac{p+1}p(\R^N)}: \|u\|_{L^{q+1}(\R^N)} = 1 \right\}
&= \inf_{u \in \dot W^{2,\frac{p+1}{p}}(\R^N) \setminus \{0\}} \frac{\int_{\R^N} |\Delta u|^{\frac{p+1}p}}{(\int_{\R^N} |u|^{q+1})^{\frac{p+1}{p(q+1)}}}.
\end{aligned}
\end{equation}
As shown by Alvino et al.~\cite{AlLiTr} (see also \cite[Corollary I.2]{Li}), it is always radially symmetric and decreasing in $r = |x|$, after a suitable translation.
Moreover, Wang in \cite[Lemma ~3.2]{Wa} and Hulshof and Van der Vorst in \cite[Theorem ~1]{HuvdV} proved that a ground state solution of \eqref{eq-LEs} is unique up to scalings.

\medskip

The present paper deals with the non-degeneracy for the critical Lane--Emden system.
Let $(U,V)$ be a ground state solution to system \eqref{eq-LEs}.
The invariance of the system under dilations and translations leads to natural solutions of the linearized system around $(U,V)$. More precisely, the functions
$$ (U_{\delta,\xi}(x),V_{\delta,\xi}(x)) := \left(\delta^{\frac{2(p+1)}{pq-1}} U(\delta(x-\xi)), \delta^{\frac{2(q+1)}{pq-1}} V(\delta(x-\xi))\right) \quad \hbox{for any}\ \delta>0,\ \xi\in\R^N$$
are solutions to system \eqref{eq-LEs}. Hence, if we differentiate the system
$$\begin{cases}
-\Delta U_{\delta,\xi} = V^p_{\delta,\xi} &\text{in } \R^N,\\
-\Delta V_{\delta,\xi} = U^q_{\delta,\xi} &\text{in } \R^N
\end{cases}
$$
with respect to the parameters at $(\delta,\xi)=(1,0)$,
we immediately see that the $(N+1)$ linearly independent functions
\begin{equation}\label{z0}\left(\Psi_0(x), \Phi_0(x)\right):= \left(x\cdot\nabla U + \frac {2(p+1)}{pq-1} U, x\cdot\nabla V + \frac {2(q+1)}{pq-1} V\right)\end{equation}
and
\begin{equation}\label{zi}\left(\Psi_i(x), \Phi_i(x)\right):= \left(\frac{\partial U}{\partial x_i}, \frac{\partial V}{\partial x_i}\right)
\quad \text{for } i=1,\dots,N
\end{equation}
solve the linear system
\begin{equation}\label{eq-lin}\begin{cases}
-\Delta \Psi = p \, V^{p-1}\Phi &\text{in } \R^N,\\
-\Delta \Phi = q \, U^{q-1}\Psi &\text{in } \R^N.\\
\end{cases}
\end{equation}

A fundamental question regarding the linear system \eqref{eq-lin} is to classify all its solutions which vanish, in a certain sense, at infinity.
Notably, one can ask if all such solutions of \eqref{eq-lin} result from the invariance of \eqref{eq-LEs}.
Such a property, which we call the {\em non-degeneracy} for system \eqref{eq-LEs}, is a key ingredient
in analyzing the blow-up phenomena of solutions to various elliptic systems on bounded or unbounded domains in $\R^N$ or Riemannian manifolds whose asymptotic behavior is encoded in \eqref{eq-LEs}.
It also plays a crucial role in building new types of bubbling solutions to the Lane--Emden systems as well as their parabolic and hyperbolic counterparts.

In this paper, we provide an affirmative answer to the question mentioned earlier,
by proving the non-degeneracy for the critical Lane--Emden system \eqref{eq-LEs} for all dimensions $N \ge 3$ and all possible pairs $(p,q)$.

Here is the precise description of our main result.

\begin{theorem}\label{thm:main}
Suppose that $N \ge 3$, $p,q>0$, $(p,q)$ satisfies \eqref{eq-hyp}, and $(U,V)$ is a ground state solution to \eqref{eq-LEs}.
Then all the solutions $(\Psi,\Phi)\in \dot W^{2,\frac{p+1}{p}}(\R^N) \times \dot W^{2,\frac{q+1}{q}}(\R^N)$
to \eqref{eq-lin} are linear combinations of $\left(\Psi_i , \Phi_i \right),$ $i=0,1,\dots,N.$
\end{theorem}

In fact, we may drop the condition $\Phi \in \dot W^{2,\frac{q+1}{q}}(\R^N)$ in the statement,
because the assumption that $\Psi \in \dot W^{2,\frac{p+1}{p}}(\R^N)$ implies this; see Subsection \ref{subsec:amd} for more comments.

In order to prove Theorem \ref{thm:main}, we perform an angular momentum decomposition.
Namely, we decompose the linear system \eqref{eq-lin} and its solutions into spherical harmonics.
Because our natural function space is not a Hilbert space such as $\dot W^{1,2}(\R^N)$, the step to determine relevant function spaces is somewhat tricky. Once it is done, we carefully study the corresponding radial parts by employing delicate ODE techniques.

\medskip

Furthermore, by using the precise decay estimate of a ground state solution to \eqref{eq-LEs} due to Hulshof and Van der Vorst \cite[Theorem 2]{HuvdV} and the maximum principle, one can prove the following lemma.

\begin{lemma}\label{decene}
Suppose that $N \ge 3$, $p,q>0$, $(p,q)$ satisfies \eqref{eq-hyp}, and $(U,V)$ is a ground state solution to \eqref{eq-LEs}. Let $(\Psi,\Phi)\in \dot W^{2,\frac{p+1}{p}}(\R^N) \times \dot W^{2,\frac{q+1}{q}}(\R^N)$ be a weak solution to \eqref{eq-lin} with
$$
\lim_{|x| \to \infty} (\Psi(x), \Phi(x)) = 0 \,.
$$
Then $(\Psi,\Phi)\in \dot W^{2,\frac{p+1}{p}}(\R^N) \times \dot W^{2,\frac{q+1}{q}}(\R^N)$.
\end{lemma}

Combining this fact and Theorem \ref{thm:main}, we deduce the following result which, we believe, is also of practical use.

\begin{corollary}\label{cor:main}
Suppose that $N \ge 3$, $p,q>0$, $(p,q)$ satisfies \eqref{eq-hyp}, and $(U,V)$ is a ground state solution to \eqref{eq-LEs}.
Then all the weak solutions $(\Psi,\Phi)$  to \eqref{eq-lin} such that $\lim_{|x| \to \infty} (\Psi(x), \Phi(x)) = 0$
are linear combinations of $\left(\Psi_i , \Phi_i \right),$ $i=0,1,\dots,N.$
\end{corollary}

The rest of the paper is devoted to the proof of Theorem \ref{thm:main} and Lemma \ref{decene}.

%%%%%%%%%%%%%%%%%%

\section{Proof of Theorem \ref{thm:main}}

\subsection{Angular momentum decomposition}\label{subsec:amd}

We write
$$
U(x) = u(|x|) \,,
\qquad
V(x) = v(|x|) \,,
$$
so that the PDE system \eqref{eq-LEs} becomes the ODE system
\begin{equation}
\label{eq:eqsystode}
-u'' - \frac{N-1}{r} u' = v^p \,,
\qquad
-v'' - \frac{N-1}{r} v' = u^q
\qquad\text{in}\ (0,\infty) \,.
\end{equation}
Moreover, since $U$ and $V$ are regular on $\R^N$, the values $u(0)$ and $v(0)$ are finite and
$$
u'(0) = v'(0)= 0 \,.
$$

Since $U$ and $V$ are radial, we can make a partial wave decomposition of \eqref{eq-lin}, that is, write
$$
\Psi(x) = \sum_{\ell=0}^\infty \sum_{m\in\mathcal M_{\ell,N}} \Psi_{\ell,m}(|x|) Y_{\ell,m}(x/|x|) \,,
\qquad
\Phi(x) = \sum_{\ell=0}^\infty \sum_{m\in\mathcal M_{\ell,N}} \Phi_{\ell,m}(|x|) Y_{\ell,m}(x/|x|) \,,
$$
where $Y_{\ell,m}$ is a basis of spherical harmonics in $L^2(\Sph^{N-1})$. The parameter $\ell\in\N_0$ is the degree of the spherical harmonic (`angular momentum' in physics terminology) and the parameter $m$ from the index set $\mathcal M_{\ell,N}$ labels the degeneracy. In the following it will only be important that
$$
\#\mathcal M_{0,N} = 1
\qquad\text{and}\qquad
\#\mathcal M_{1,N} = N \,,
$$
as well as that $Y_{0,0}$ is a constant function and that $\spa\{ Y_{1,m} :\ m\in\mathcal M_{1,N}\}$ coincides with the span of the coordinate functions $x_n/|x|$, $n=1,\ldots,N$. For each $\ell$ and $m$, the pair of functions $(\Psi_{\ell,m},\Phi_{\ell,m})$ satisfies the following equations, where, for simplicity, we write $(\psi,\phi)$ instead of $(\Psi_{\ell,m},\Phi_{\ell,m})$,
\begin{align}
\label{eq:eqlinsystode1}
-\psi'' - \frac{N-1}{r}\psi' + \frac{\ell(\ell+N-2)}{r^2}\psi & = p\, v^{p-1}\phi
\qquad\text{in}\ (0,\infty) \,, \\
\label{eq:eqlinsystode2}
-\phi'' - \frac{N-1}{r}\phi' + \frac{\ell(\ell+N-2)}{r^2}\phi & = q\, u^{q-1}\psi
\qquad\text{in}\ (0,\infty) \,.
\end{align}
We have
\begin{equation}
\label{eq:eqlinsystodebc}
\lim_{r\to 0} r^{-\ell} \psi(r)
\qquad\text{and}\qquad
\lim_{r\to 0} r^{-\ell} \phi(r)
\qquad\text{exist}
\end{equation}
and
\begin{equation}
\label{eq:eqlinsystodebc1}
\lim_{r\to 0} (r^{-\ell} \psi)'(r) =
\lim_{r\to 0} (r^{-\ell} \phi)'(r) =0 \,.
\end{equation}

Finally, let us comment on the relevant function spaces. We are concerned with solutions $\Psi\in\dot W^{2,\frac{p+1}{p}}(\R^N)$ and then \eqref{eq-lin} and $U,\Psi\in L^{q+1}(\R^N)$ (by Sobolev) implies that $\Phi\in \dot W^{2,\frac{q+1}{q}}(\R^N)$. Let us deduce corresponding properties of the $\Psi_{\ell,m}$ and $\Phi_{\ell,m}$. We denote by $\mathcal E_\ell^s$ the completion of $r^\ell C^2_c[0,\infty)$ with respect to
$$
\|f\|_{\mathcal E^s_\ell} := \left( \int_0^\infty \left|\left(f'' + \frac{N-1}{r}f' - \frac{\ell(\ell+N-2)}{r^2}f \right)\right|^s r^{N-1}\,dr \right)^{\frac 1s} \,.
$$
(We suppress $N$ from the notation of $\mathcal E_\ell^s$ for the sake of simplicity.) We claim that
$$
\Psi_{\ell,m}\in \mathcal E_\ell^\frac{p+1}{p} \,,
\qquad
\Phi_{\ell,m}\in \mathcal E_\ell^\frac{q+1}{q} \,.
$$
Indeed, if the $Y_{\ell,m}$ are normalized in $L^2(\Sph^{N-1})$, then
$$
\Psi_{\ell,m}(r) = \int_{\Sph^{N-1}} \overline{Y_{\ell,m}(\omega)} \Psi(r\omega)\,d\omega
$$
and
$$
\Psi_{\ell,m}'' + \frac{N-1}{r}\Psi_{\ell,m}' - \frac{\ell(\ell+N-2)}{r^2}\Psi_{\ell,m} = \int_{\Sph^{N-1}} \overline{Y_{\ell,m}(\omega)} (\Delta\Psi)(r\omega)\,d\omega \,.
$$
Thus, by H\"older's inequality on $\Sph^{N-1}$,
\begin{align*}
& \| \Psi_{\ell,m} \|_{\mathcal E_\ell^{\frac{p+1}{p}}} \leq \|\Delta\Psi\|_{L^\frac{p+1}{p}(\R^N)} \|Y_{\ell,m}\|_{L^{p+1}(\Sph^{N-1})} \,.
\end{align*}
A similar argument shows
\begin{align*}
& \| \Phi_{\ell,m} \|_{\mathcal E_\ell^{\frac{q+1}{q}}} \leq \|\Delta\Phi\|_{L^\frac{q+1}{q}(\R^N)} \|Y_{\ell,m}\|_{L^{q+1}(\Sph^{N-1})} \,.
\end{align*}

In view of the above fact, Theorem \ref{thm:main} is an immediate consequence of the following proposition.

\begin{proposition}\label{mainprop}
Let $(\psi,\phi)\in\mathcal E_\ell^{\frac{p+1}{p}}\times \mathcal E_\ell^{\frac{q+1}{q}}$ be a solution of \eqref{eq:eqlinsystode1}, \eqref{eq:eqlinsystode2} satisfying \eqref{eq:eqlinsystodebc} and \eqref{eq:eqlinsystodebc1}.
\begin{enumerate}
\item[(a)] If $\ell=0$, then $(\psi,\phi)$ is a multiple of $(ru' + \frac{2(p+1)}{pq-1} u, rv'+ \frac{2(q+1)}{pq-1} v)$.
\item[(b)] If $\ell=1$, then $(\psi,\phi)$ is a multiple of $(u',v')$.
\item[(c)] If $\ell\geq 2$, then $(\psi,\phi)\equiv 0$.
\end{enumerate}
\end{proposition}

We will prove the proposition in the following two subsections, which will deal with the two different aspects of this result. On the one hand, for $\ell=0,1$ we need to show that there are no other solutions than the known ones. This is proved using a uniqueness result in Lemma \ref{unique}. On the other hand, for $\ell\geq 2$ we need to prove that there are no non-trivial finite energy solutions at all. This is proved by adapting and completing an argument from \cite{LuWe} for the special case $p=1$, $q=\frac{N+4}{N-4}$.

%%%%%%%%%%%%%%%%%%

\subsection{A uniqueness theorem}

Our first goal will be to prove the following uniqueness result.

\begin{lemma}\label{unique}
Let $(\psi,\phi)$ be a solution to \eqref{eq:eqlinsystode1}, \eqref{eq:eqlinsystode2} with $\lim_{r\to 0} r^{-\ell}\psi(r)=0$ and $\psi\not\equiv 0$. Then $r^{-\ell}\psi$ is strictly monotone. In particular, $\psi$ is strictly monotone and $\psi\not\in\mathcal E_\ell^{\frac{p+1}{p}}$.
\end{lemma}

Before giving the proof, let us apply it to prove the first half of Proposition \ref{mainprop}. We shall often use the fact (see, e.g., \cite[Cor.~I.2]{Li})
that
\begin{equation}
\label{eq:gspos}
u > 0
\qquad\text{and}\qquad
v > 0
\qquad\text{in}\ [0,\infty) \,.
\end{equation}
and
\begin{equation}
\label{eq:gsmono}
u' < 0
\qquad\text{and}\qquad
v'<0
\qquad\text{in}\ (0,\infty) \,.
\end{equation}

\begin{proof}[Proof of Proposition \ref{mainprop}. Parts (a) and (b).]
  Note that $(\Psi_0,\Phi_0)$ defined in \eqref{z0} corresponds to angular momentum $\ell=0$, while $(\Psi_i,\Phi_i)$ for $i=1,\dots,N$ correspond to angular momentum $\ell=1$. By the above discussion, this means that $(ru' + \frac{2(p+1)}{pq-1} u, rv'+ \frac{2(q+1)}{pq-1} v)$ is a solution of \eqref{eq:eqlinsystode1}, \eqref{eq:eqlinsystode2} with $\ell=0$ and $(u',v')$ is a solution of \eqref{eq:eqlinsystode1}, \eqref{eq:eqlinsystode2} for $\ell=1$. Moreover, since the right sides of the equations satisfy the corresponding integrability conditions, we have $(ru' + \frac{2(p+1)}{pq-1} u, rv'+ \frac{2(q+1)}{pq-1} v)\in \mathcal E_0^{\frac{p+1}{p}}\times \mathcal E_0^{\frac{q+1}{q}}$ and $(u',v')\in \mathcal E_1^{\frac{p+1}{p}}\times \mathcal E_1^{\frac{q+1}{q}}$.

It remains to be proved that these are the only solutions. We first assume that $\ell=0$. If $(\psi,\phi)\in\mathcal E_0^{\frac{p+1}{p}}\times \mathcal E_0^{\frac{q+1}{q}}$ is a solution of \eqref{eq:eqlinsystode1}, \eqref{eq:eqlinsystode2}, then
$$
\left(\psi,\phi \right) - \frac{pq-1}{2(p+1)}\,\frac{\psi(0)}{u(0)}\, \left(ru' + \frac{2(p+1)}{pq-1} u,\ rv'+ \frac{2(q+1)}{pq-1} v \right)
$$
is a solution of \eqref{eq:eqlinsystode1}, \eqref{eq:eqlinsystode2} in $\mathcal E_0^{\frac{p+1}{p}}\times \mathcal E_0^{\frac{q+1}{q}}$ whose first component vanishes at zero. (Note that here we use $u(0)\neq 0$ which follows from \eqref{eq:gspos}.) Thus, by Lemma \ref{unique}, the above solution vanishes identically, which means that $(\psi,\phi)$ is a multiple of $(ru' + \frac{2(p+1)}{pq-1} u, rv'+ \frac{2(q+1)}{pq-1} v)$.

The proof for $\ell=1$ is similar, except that now we use the fact that $u'(0)=0$ and, by equation \eqref{eq:eqsystode} and \eqref{eq:gspos},
$$
u''(0) = - N^{-1} v(0)^p \neq 0 \,.
$$
Thus, if $(\psi,\phi)\in\mathcal E_1^{\frac{p+1}{p}}\times \mathcal E_1^{\frac{q+1}{q}}$ is a solution of \eqref{eq:eqlinsystode1}, \eqref{eq:eqlinsystode2}, then
$$
(\psi,\phi) - \frac{\psi'(0)}{u''(0)}\, (u', v')
$$
is a solution in $\mathcal E_1^{\frac{p+1}{p}}\times \mathcal E_1^{\frac{q+1}{q}}$ whose first component is $o(r)$ at the origin. Thus, by the lemma the above solution vanishes identically, which means that $(\psi,\phi)$ is a multiple of $(u',v')$. This completes the proof of parts (a) and (b) in the proposition.
\end{proof}

\begin{proof}[Proof of Lemma \ref{unique}]
\emph{Step 1. The case $\ell=0$.} Since zeros of solutions of ordinary differential equations cannot accumulate at a finite point, we know that $\psi$ is either positive or negative in a right neighborhood of zero. By multiplying both $\psi$ and $\phi$ by $-1$ if necessary, we may assume that $\psi$ is positive in a right neighborhood of zero. Let
$$
R:= \sup\{ r>0 :\ \psi>0 \ \text{in}\ (0,r) \} \,,
$$
so, by assumption, $R>0$. We will show that $\psi$ is strictly increasing in $(0,R)$. Note that this implies, in particular, that $R=\infty$, because otherwise we had $\psi(R)=0$ and then $0=\psi(R)-\psi(0)=\int_0^R \psi'(r)\,dr >0$, a contradiction.

Writing the equation for $\phi$ as $(r^{N-1}\phi')' = - q r^{N-1} u^{q-1}\psi$ and using the fact that $\lim_{r\to 0} r^{N-1}\phi'(r)=0$, we obtain
$$
r^{N-1} \phi'(r) = -q \int_0^r u(s)^{q-1}\psi(s) s^{N-1}\,ds \,.
$$
By \eqref{eq:gspos}, this proves that $\phi'<0$ in $(0,R)$. We now deduce from the equation for $\psi$ that
$$
- N \psi''(0) = \lim_{r\to 0} \left(-\psi''(r) - \frac{N-1}{r}\psi'(r)\right) = \lim_{r\to 0} p\, v(r)^{p-1}\phi(r) = p\, v(0)^{p-1}\phi(0) \,.
$$
Since $v(0)>0$ (again from \eqref{eq:gspos}) and $\psi''(0)\geq 0$ (this follows from the fact that $\psi$ is positive in a right neighborhood of zero and $\psi(0)=\psi'(0)=0$), we conclude that $\phi(0)\leq 0$. This, together with the fact that $\phi'<0$ in $(0,R)$ implies that $\phi<0$ in $(0,R)$.

Now writing the equation for $\psi$ as $(r^{N-1}\psi')' = -p r^{N-1} v^{p-1}\phi$ and using the fact that $\lim_{r\to 0} r^{N-1}\psi'(r)=0$, we obtain
$$
r^{N-1} \psi'(r) = - p \int_0^r v(s)^{p-1}\phi(s) s^{N-1}\,ds \,.
$$
By \eqref{eq:gspos} and $\phi<0$ in $(0,R)$, we conclude that $\psi'>0$ in $(0,R)$, as claimed.

\medskip

\emph{Step 2. The case $\ell\geq 1$.} We use a standard trick to reduce the case $\ell\geq 1$ to the case $\ell=0$ by increasing $N$. Let $\tilde\psi(r) := r^{-\ell}\psi(r)$ and $\tilde\phi(r):=r^{-\ell}\phi(r)$ and note that
\begin{align}
\label{eq:eqlinsystodealt}
-\tilde\psi'' - \frac{N+2\ell-1}{r}\tilde\psi' & = p\, v^{p-1}\tilde\phi
\qquad\text{in}\ (0,\infty) \,. \\
-\tilde\phi'' - \frac{N+2\ell-1}{r}\tilde\phi' & = q\, u^{q-1}\tilde\psi
\qquad\text{in}\ (0,\infty) \,.
\end{align}
Moreover, we have
$$
\tilde\psi'(0) = \tilde\phi'(0)
\qquad\text{and}\qquad
\lim_{r\to 0} \tilde\phi(r)
\ \text{exists}
$$
and, by the assumption of the lemma,
$$
\tilde\psi(0)=0 \,.
$$
Therefore, from Step 1 we infer that $\tilde\psi= r^{-\ell}\psi$ is strictly monotone. Since it vanishes at the origin, this implies, in particular, that $r^{-\ell} \psi$ has the same sign as $(r^{-\ell}\psi)'$. Thus, $\psi' = r^\ell (r^{-\ell}\psi)' + \ell r^{-1} \psi$ has a fixed sign, which means that $\psi$ is strictly monotone, as claimed.
\end{proof}

%%%%%%%%%%%%%%%%%%

\subsection{An identity for solutions and its consequences}

The following two functions will play an important role in what follows,
$$
I_1(r) := r^{N-1} (v''(r)\psi(r) - v'(r)\psi'(r)) \,,
\qquad
I_2(r) := r^{N-1} (u''(r)\phi(r) - u'(r)\phi'(r)) \,.
$$
In the next lemma we compute their derivatives and prove an integral representation for their sum.

\begin{lemma}\label{poho}
For any $r>0$,
\begin{align}
\label{eq:i1der}
I_1'(r) & = r^{N-1} \left( -q u^{q-1}u'\psi + p v^{p-1} v'\phi - \frac{\ell(\ell+N-2)-(N-1)}{r^2}v'\psi \right), \\
\label{eq:i2der}
I_2'(r) & = r^{N-1} \left( - p v^{p-1} v'\phi + q u^{q-1}u'\psi - \frac{\ell(\ell+N-2)-(N-1)}{r^2}u'\phi \right).
\end{align}
In particular, for any $R>0$,
\begin{equation}
\label{eq:poho}
I_1(R) + I_2(R) = - \int_0^R \frac{\ell(\ell+N-2)-(N-1)}{r^2} ( u'(r)\phi(r)+ v'(r)\psi(r))r^{N-1}\,dr \,.
\end{equation}
\end{lemma}

In the proof of part (a) of Proposition \ref{mainprop} we have already shown that $(u',v')$ solves \eqref{eq:eqlinsystode1}, \eqref{eq:eqlinsystode2} with $\ell=1$. For easier reference we record the equation (which is obtained from \eqref{eq:eqsystode} by differentiation),
\begin{align}
\label{eq:eqlinsystodezm1}
-u''' - \frac{N-1}{r}u'' + \frac{N-1}{r^2}u' & = p v^{p-1}v'
\qquad\text{in}\ (0,\infty) \,, \\
\label{eq:eqlinsystodezm2}
-v''' - \frac{N-1}{r}v'' + \frac{N-1}{r^2}v' & = q u^{q-1}u'
\qquad\text{in}\ (0,\infty) \,.
\end{align}

\begin{proof}[Proof of Lemma \ref{poho}]
Using the equations \eqref{eq:eqlinsystodezm2} and \eqref{eq:eqlinsystode1} for $v'$ and $\psi$,
\begin{align*}
I_1'(r) & = r^{N-1} \left( (v''' + \frac{N-1}{r}v'')\psi - v'(\psi''+ \frac{N-1}{r}\psi') \right) \\
& = r^{N-1} \left( (- q u^{q-1}u' + \frac{N-1}{r^2}v') \psi  - v'( -p v^{p-1}\phi + \frac{\ell(\ell+N-2)}{r^2} \psi) \right) \\
& = r^{N-1} \left( -q u^{q-1}u'\psi + p v^{p-1} v'\phi - \frac{\ell(\ell+N-2)-(N-1)}{r^2}v'\psi \right).
\end{align*}
Similarly, using the equations \eqref{eq:eqlinsystodezm1} and \eqref{eq:eqlinsystode2} for $u'$ and $\phi$,
\begin{align*}
I_2'(r) & = r^{N-1} \left( (u'''+\frac{N-1}{r}u'')\phi - u' (\phi'' + \frac{N-1}{r}\phi') \right) \\
& = r^{N-1} \left( (-p v^{p-1}v' + \frac{N-1}{r^2} u')\phi - u' (-q u^{q-1}\psi + \frac{\ell(\ell+N-2)}{r^2}\phi) \right) \\
& = r^{N-1} \left( - p v^{p-1} v'\phi + q u^{q-1}u'\psi - \frac{\ell(\ell+N-2)-(N-1)}{r^2}u'\phi \right).
\end{align*}
This proves the first two formulas in the lemma. To prove the third one, we add the first two and integrate them between $0$ and $R$.
\end{proof}

We finally turn to the

\begin{proof}[Proof of Proposition \ref{mainprop}. Part (c)]
Our goal is to prove that if $(\psi,\phi)$ solves \eqref{eq:eqlinsystode1}, \eqref{eq:eqlinsystode2} for $\ell\geq 2$, then $(\psi,\phi)\equiv 0$. To prove this, it suffices to show that $\psi\equiv 0$. To prove the latter, we argue by contradiction and assume $\psi\not\equiv 0$. As in the proof of Lemma \ref{unique}, we may assume, without loss of generality, that $\psi$ is positive in a right neighborhood of zero. Let
$$
r_1 :=\sup\{ r>0 :\ \psi>0 \ \text{in}\ (0,r) \} \,,
$$
so, by assumption, $r_1>0$. Moreover, if $r_1<\infty$, then
\begin{equation}
\label{eq:psibdry}
\psi(r_1) = 0
\qquad\text{and}\qquad
\psi'(r_1)\leq 0 \,.
\end{equation}

We claim that $\phi$ is positive in a right neighborhood of zero. Since $\phi(0)=0$ (because $\ell>0$), this is a consequence of the following two facts,
\begin{enumerate}
\item $\phi$ has no negative local minimum in $(0,r_1)$.
\item $\phi$ takes a positive value in $(0,r_1)$.
\end{enumerate}
Item (1) follows from the equation \eqref{eq:eqlinsystode2} for $\phi$, since by \eqref{eq:gspos} at a negative local minimum the left side would be negative whereas the right side is positive in $(0,r_1)$. To prove item (2) note that $\psi$ has a positive local maximum in $(0,r_1)$ (since $\psi(0)=0$ and $\lim_{r\to r_1}\psi(r)=0$). Evaluating the equation for $\psi$ at this point, we see that the left side is positive and therefore so is the right side. Thus, by \eqref{eq:gspos}, $\phi$ at this point is positive.

Because of the preceding arguments
$$
r_2:= \sup\{ r>0:\ \phi>0 \ \text{in}\ (0,r)\}
$$
is positive. Clearly, if $r_2<\infty$, then
\begin{equation}
\label{eq:phibdry}
\phi(r_2) = 0
\qquad\text{and}\qquad
\phi'(r_2)\leq 0 \,.
\end{equation}

Next, we show that
\begin{equation}
\label{eq:psiphioutside}
\psi<0 \ \text{in}\ (r_1,r_2) \ \text{if}\ r_1<r_2 \,,
\qquad
\phi<0 \ \text{in}\ (r_2,r_1) \ \text{if}\ r_2<r_1 \,.
\end{equation}
We begin with the second assertion. We first argue that $\phi$ is negative in a right neighborhood of $r_2$. Recall from \eqref{eq:phibdry} that $\phi'(r_2)\leq 0$. The negativity in a right neighborhood is clear if $\phi'(r_2)<0$, while if $\phi'(r_2)=0$, the equation \eqref{eq:eqlinsystode2} for $\phi$ evaluated at $r_2$, together with the fact that $\psi(r_2)>0$, implies $\phi''(r_2)<0$, which again implies the negativity in a right neighborhood. The negativity in the whole interval $(r_2,r_1)$ now follows from item (1) above. The assertion for $\psi$ follows from the same arguments.

\medskip

After these preliminaries we now turn to the main part of the proof of part (c) of Proposition \ref{mainprop}. We first assume that $\min\{r_1,r_2\}<\infty$ and choose $R=\min\{r_1,r_2\}$ in identity \eqref{eq:poho} in Lemma \ref{poho}. Note that with this choice, using $\ell\geq 2$ and \eqref{eq:gsmono},
$$
\int_0^R \frac{\ell(\ell+N-2)-(N-1)}{r^2} ( u'(r)\phi(r)+ v'(r)\psi(r))r^{N-1}\,dr <0 \,.
$$
We now show that for the above choice of $R$,
\begin{equation}
\label{eq:iproof}
I_1(R) + I_2(R) \leq 0 \,,
\end{equation}
which will lead to the desired contradiction.

It is easy to see that
\begin{equation}
\label{eq:iproof1}
I_1(R) \leq 0
\qquad\text{if}\ R=r_1
\qquad\text{and}\qquad
I_2(R) \leq 0
\qquad\text{if}\ R=r_2 \,.
\end{equation}
Indeed, if $R=r_1$, then, by \eqref{eq:gsmono} and \eqref{eq:psibdry}, $I_1(R) = - r_1^{N-1} v'(r_1)\psi'(r_1)\leq 0$ and if $R=r_2$, then, by \eqref{eq:phibdry}, $I_2(R) = -r_2^{N-1}u'(r_2)\phi'(r_2) \leq 0$.

We now show that
\begin{equation}
\label{eq:iproof2}
I_2(R) \leq 0
\qquad\text{if}\ R=r_1
\qquad\text{and}\qquad
I_1(R) \leq 0
\qquad\text{if}\ R=r_2 \,.
\end{equation}
Note that in case $r_1=r_2$ this follows from the previous assertion, so we may assume that $r_1\neq r_2$ (and we continue to assume that $\min\{r_1,r_2\}<\infty$). In order to prove the first assertion in \eqref{eq:iproof2}, let $r_1<r_2$. Using \eqref{eq:psiphioutside}, \eqref{eq:gsmono} and $\ell\geq 2$, we see that each one of the three terms in the parenthesis on the right side of \eqref{eq:i2der} is positive in $(r_1,r_2)$. Thus, $I_2'>0$ in $(r_1,r_2)$ and therefore
$$
I_2(r_1)<I_2(r_2)\leq 0
$$
where the second inequality follows from \eqref{eq:iproof1}.

Similarly, in order to prove the second assertion in \eqref{eq:iproof2}, let $r_2<r_1$. Using \eqref{eq:psiphioutside}, \eqref{eq:gsmono} and $\ell\geq 2$, we see that each one of the three terms in the parenthesis on the right side of \eqref{eq:i1der} is positive in $(r_2,r_1)$. Thus, $I_1'>0$ in $(r_2,r_1)$ and therefore
$$
I_1(r_2)<I_1(r_1)\leq 0
$$
where the second inequality follows from \eqref{eq:iproof1}.

This completes the proof of \eqref{eq:iproof2} and therefore the proof of \eqref{eq:iproof}.

\medskip

We still need to deal with the case $\min\{r_1,r_2\}=\infty$, that is, $r_1=r_2=\infty$. We let $R\to\infty$ in \eqref{eq:poho}. Since the integrand on the right side is negative, the left side converges as $R\to\infty$ either to $+\infty$ or to a positive number. The following lemma implies that the left side converges, along a subsequence, to $0$, which is again a contradiction and concludes the proof of part (c) of Proposition \ref{mainprop}.
\end{proof}

In the previous proof we used the following fact.

\begin{lemma}
If $(\psi,\phi)\in\mathcal E^\frac{p+1}{p}\times\mathcal E^\frac{q+1}{q}$, then
$$
\liminf_{R\to\infty} |I_1(R) + I_2(R)|= 0 \,.
$$
\end{lemma}

There are several possible proofs of this lemma. One possibility would be a detailed ODE analysis giving the precise asymptotics of $\psi$, $\phi$ and their derivatives at infinity. We have chosen a softer approach, based only on the finite energy assumption, together with Sobolev embedding theorems. More precisely, we shall use the inequalities
\begin{align}
\label{eq:sob1}
\| f \|_{\mathcal E_\ell^\frac{p+1}{p}} & \gtrsim \left( \int_0^\infty |f|^{q+1} \,r^{N-1}\,dr \right)^\frac{1}{q+1} \,, \\
\label{eq:hardy}
\| f \|_{\mathcal E_\ell^\frac{q+1}{q}} & \gtrsim
\left( \int_0^\infty r^{-\frac{q+1}{q}} |f'|^\frac{q+1}{q} r^{N-1}\,dr \right)^\frac{q}{q+1} \,, \\
\label{eq:sob2}
\| f \|_{\mathcal E_\ell^\frac{p+1}{p}} & \gtrsim \left( \int_0^\infty |f'|^t \,r^{N-1}\,dr \right)^\frac{1}{t} \,, \qquad \frac{1}{t} = \frac{p}{p+1}-\frac{1}{N} \,, \\
\label{eq:sob3}
\| f \|_{\mathcal E_\ell^\frac{q+1}{q}} & \gtrsim \left( \int_0^\infty |f'|^s \,r^{N-1}\,dr \right)^\frac{1}{s} \,, \qquad \frac{1}{s} = \frac{q}{q+1}-\frac{1}{N} \,.
\end{align}

Inequality \eqref{eq:sob1} follows from the Sobolev inequality $\|\Delta F\|_{L^\frac{p+1}{p}(\R^N)} \gtrsim \|F\|_{L^{q+1}(\R^N)}$, applied to $F(x)=f(|x|)Y_{\ell,m}(x/|x|)$. Similarly, inequalities \eqref{eq:hardy}, \eqref{eq:sob2} and \eqref{eq:sob3} follow from Hardy and Sobolev inequalities, bounding $|\nabla F|\geq |f'| |Y_{\ell,m}|$. Note that Sobolev's inequality is applicable since $\frac{q}{q+1} > \frac1N$ and $\frac{p}{p+1} > \frac1N$. Indeed, the latter are equivalent to $\frac{1}{q+1}<\frac{N-1}{N}$ and $\frac{1}{p+1}<\frac{N-1}{N}$, and these inequalities are valid since the scaling relation \eqref{eq-hyp} implies that $\frac{1}{q+1}<\frac{N-2}{N}$ and $\frac{1}{p+1}<\frac{N-2}{N}$.

\begin{proof}
We show that
$$
\int_0^\infty \left( |I_1(r)| + |I_2(r)| \right)dr <\infty \,,
$$
which clearly implies the assertion. We prove this only for $I_1$, the argument $I_2$ being similar.

Using the equation for $v$ we write
$$
I_1(r) = r^{N-1} \left(-u(r)^q \psi(r) - \frac{N-1}{r}v'(r)\psi(r) - v'(r)\psi'(r) \right)
$$
and show that all three terms on the right side are separately integrable. By \eqref{eq:sob1} with $f=u,\psi$, we have $u\in L^{q+1}(\R_+,r^{N-1}\,dr)$ and $\psi\in L^{q+1}(\R_+,r^{N-1}\,dr)$ and therefore $u^q\psi\in L^1(\R_+,r^{N-1}\,dr)$, which means $r^{N-1} u^q\psi\in L^1(\R_+)$. Moreover, $\Delta V = -U^q \in L^{\frac{q+1}q}(\R^N)$ and therefore, by \eqref{eq:hardy} with $f= v$, we have $r^{-1} v' \in L^\frac{q+1}{q}(\R_+,r^{N-1}\,dr)$. Together with $\psi\in L^{q+1}(\R_+,r^{N-1}\,dr)$ this implies $r^{N-2}v'\psi\in L^1(\R_+)$. Finally, \eqref{eq:sob2} with $f=\psi$ and \eqref{eq:sob3} with $f=v$ imply that $\psi'\in L^t(\R_+,r^{N-1}\,dr)$ and $v'\in L^s(\R_+,r^{N-1}\,dr)$, where $\frac1s = \frac{q}{q+1} - \frac1N$ and $\frac1t = \frac{p}{p+1} - \frac1N$.  Moreover, \eqref{eq-hyp} implies that $\frac{1}{s} + \frac{1}{t}=1$ and therefore $v'\psi'\in L^1(\R_+,r^{N-1}\,dr)$, which means $r^{N-1}v'\psi'\in L^1(\R_+)$. This completes the proof.
\end{proof}

As we mentioned before, the basic idea for treating the case $\ell\geq 2$ comes from \cite{LuWe}. However, we do not see where the case $r_1=r_2=\infty$ is handled in that paper.

%%%%%%%%%%%%

\section{Proof of Lemma \ref{decene}}

\subsection{Scheme of the proof}

After interchanging the roles of $p$ and $q$ if necessary, we may and will assume that $q \geq p$. Thus, $\frac{2}{N-2} < p \le \frac{N+2}{N-2} \leq q$.

We will prove Lemma \ref{decene} by a repeated application of the maximum principle, using the asymptotic behavior of the ground state, which we quote from \cite[Theorem 2]{HuvdV}.

\begin{lemma}
For each $\frac{2}{N-2} < p \le \frac{N+2}{N-2}$, there are positive constants $a_p$ and $b_p$ such that
\begin{equation}\label{eq_dec}
\lim_{r \to \infty} r^{N-2}\, v(r) = b_p
\quad \hbox{and} \quad
\begin{cases}
\lim\limits_{r \to \infty} r^{p(N-2)-2}\, u(r) = a_p &\text{if } \frac{2}{N-2} < p < \frac{N}{N-2}\,,\\
\lim\limits_{r \to \infty} \dfrac{r^{N-2}}{\log r}\, u(r) = a_p &\text{if } p = \frac{N}{N-2}\,,\\
\lim\limits_{r \to \infty} r^{N-2}\, u(r) = a_p &\text{if } \frac{N}{N-2} < p \le \frac{N+2}{N-2} \,.
\end{cases}
\end{equation}
\end{lemma}

\subsection{The case $\frac{2}{N-2} < p < \frac{N}{N-2}$}
We begin with an elementary algebraic lemma.
\begin{lemma}
If $\frac{2}{N-2} < p < \frac{N}{N-2}$, then
\begin{equation}\label{eq_ineq}
(p(N-2)-2)(q-1) > 4-(N-2)(p-1) > 2.
\end{equation}
\end{lemma}
\begin{proof}
Set $A = (p+1)(N-2) \in (N, 2(N-1))$. Note that
\begin{align*}
(N-2)(p-1) + (p(N-2)-2)(q-1) > 4 &\Leftrightarrow (A-N)(q+1) > A \\
&\Leftrightarrow \frac{A-N}{A} > \frac{1}{q+1} = (N-2)\left(\frac{A-N}{AN}\right) \\
&\Leftrightarrow 1 > \frac{N-2}{N}.
\end{align*}
Clearly, the last inequality holds for all $N \ge 3$. Thus the first inequality in \eqref{eq_ineq} is true.
The second inequality in \eqref{eq_ineq} is a direct consequence of the condition that $p > \frac{2}{N-2}$.
\end{proof}

We are now in position to prove Lemma \ref{decene} for $\frac{2}{N-2} < p < \frac{N}{N-2}$.

\begin{proof}[Proof of Lemma \ref{decene} for $\frac{2}{N-2} < p < \frac{N}{N-2}$]
The proof is divided into 3 steps.

\medskip \noindent
\textsc{Step 1.} We assert that for any pair $(\alpha, \nu)$ such that
\begin{equation}\label{eq_proof_12}
\alpha \ge 0 \quad \text{and} \quad 0 < \nu < \min\{N-2, (p(N-2)-2)(q-1)-2+\alpha\}
\end{equation}
one has
\begin{equation}\label{eq_proof_11}
|\Psi(x)| \le \frac{C}{|x|^{\alpha}} \quad \text{on } \{|x| \ge 1\}\quad \Rightarrow\quad |\Phi(x)| \le \frac{C'}{|x|^{\nu}} \quad \text{on } \{|x| \ge 1\} \,.
\end{equation}
Observe that the minimum in \eqref{eq_proof_12} is positive by virtue of \eqref{eq_ineq}.

Consider
\[G_{\nu}(x) = \Phi(x) - \frac{m_{\nu}}{|x|^{\nu}}
\quad \text{on } \{|x| \ge 1\}\]
where $m_{\nu} > 0$ is a number to be determined. If $m_\nu\geq \sup_{\{|x|=1\}}\Phi(x)$, then
\[G_{\nu}(x) \le 0 \quad \text{on } \{|x| = 1\}\]
Moreover, \eqref{eq_proof_12}, the first inequality in \eqref{eq_proof_11} and \eqref{eq_dec} show that
\begin{align*}
-\Delta G_{\nu}(x) &= q\, U^{q-1} \Psi - \frac{m_{\nu} \nu (N-2-\nu)}{|x|^{\nu+2}} \\
&\le \frac{C''}{|x|^{(p(N-2)-2)(q-1)+\alpha}} - \frac{m_{\nu} \nu (N-2-\nu)}{|x|^{\nu+2}} \le 0 \quad \text{in } \{|x| >1\}
\end{align*}
provided $m_{\nu} \geq C''/(\nu(N-2-\nu))$. The maximum principle yields that for any number $R > 1$,
\[G_{\nu}(x) \le \max_{\{|x| = R\}} (G_{\nu}(x))_+ \quad \text{on } \{1 \le |x| \le R\}.\]
Taking $R \to \infty$ and using the uniform decay assumption on $\Phi$, we deduce
\[G_{\nu}(x) \le 0, \quad \text{i.e.,} \quad \Phi(x) \le \frac{m_{\nu}}{|x|^{\nu}} \quad \text{on } \{|x| \ge 1\}.\]
By the same reasoning, we obtain a similar upper bound on $-\Phi$ in $\{|x| \ge 1\}$. This proves the assertion \eqref{eq_proof_11}.

\medskip
An analogous argument shows that for any pair $(\beta,\mu)$ such that
\begin{equation}\label{eq_proof_22}
\beta \ge 0 \quad \text{and} \quad 0 < \mu < \min\{N-2, (N-2)(p-1)-2+\beta\}
\end{equation}
one has
\begin{equation}\label{eq_proof_21}
|\Phi(x)| \le \frac{C}{|x|^{\beta}} \quad \text{on } \{|x| \ge 1\} \quad \Rightarrow \quad |\Psi(x)| \le \frac{C'}{|x|^{\mu}} \quad \text{on } \{|x| \ge 1\} \,.
\end{equation}
Unlike \eqref{eq_proof_12}, the minimum in \eqref{eq_proof_22} may be non-positive unless $\beta$ is large enough.

\medskip \noindent
\textsc{Step 2.} We assert that for any $\eta > 0$ there is a $C>0$ such that
\begin{equation}\label{eq_opt_dec}
|\Psi(x)| \le \frac{C}{1+|x|^{(N-2)p-2-\eta}}
\quad \text{and} \quad
|\Phi(x)| \le \frac{C}{1+|x|^{N-2-\eta}} \quad \text{in } \R^N.
\end{equation}

The uniform decay condition tells us that $|\Psi(x)| \le C$ in $\R^N$.
Hence, taking $\alpha =\alpha_1 = 0$ in \eqref{eq_proof_11}, we find
\[|\Phi(x)| \le \frac{C'}{|x|^{\beta_1}} \quad \text{on } \{|x| \ge 1\}\]
for any fixed $0 < \beta_1 < \min\{N-2, (p(N-2)-2)(q-1)-2\}$.
Next, taking $\beta = \beta_1$ in \eqref{eq_proof_21}, we also find
\[|\Psi(x)| \le \frac{C''}{|x|^{\alpha_2}} \quad \text{on } \{|x| \ge 1\}\]
for any fixed $0 < \alpha_2 < \min\{N-2, (N-2)(p-1) + (p(N-2)-2)(q-1)-4\}$. (By \eqref{eq_ineq}, the minimum is positive and therefore such $\alpha_2$ does exist.)
Taking $\alpha = \alpha_2$, we again employ \eqref{eq_proof_11} to update the range of $\beta$.
In this way, we can construct a (finite) sequence $\{(\alpha_n, \beta_n)\}_{n \in \N}$ such that
\begin{itemize}
\item[-] Each of $\{\alpha_n\}$ and $\{\beta_n\}$ is increasing;
\item[-] For each $n$, it holds that $\beta_n > \alpha_{n+1}$;
\item[-] $\beta_n \nearrow N-2$ as $n$ gets large. In particular, $\alpha_{n+1} \nearrow (N-2)p-2$.
\end{itemize}
By picking $(\alpha_n, \beta_n)$ such that $\alpha_n \ge (N-2)p-2-\eta$ and $\beta_n \ge (N-2)-\eta$, we conclude that \eqref{eq_opt_dec} is true.

\medskip \noindent
\textsc{Step 3.} We conclude the proof. By \eqref{eq_dec} and \eqref{eq_opt_dec},
\[\|\Delta \Psi\|_{L^{\frac{p+1}p}(\R^N)} = p \|V^{p-1} \Phi\|_{L^{\frac{p+1}p}(\R^N)}
\le C \left(\int_{\R^N} \frac{dx}{1+|x|^{(N-2)(p+1)-\eta}}\right)^{\frac p{p+1}}\]
where $\eta > 0$ can be taken arbitrarily small.
Hence we infer from the relation $(N-2)(p+1) > N$ that the above integral is finite.
This confirms that $\Psi \in \dot{W}^{2,{\frac{p+1}p}}(\R^N)$. This in turn yields that $\Phi \in \dot W^{2,\frac{q+1}{q}}(\R^N)$.
\end{proof}

\subsection{The case $\frac{N}{N-2} \le p \le \frac{N+2}{N-2}$}

In this subsection, we slightly modify the argument in the previous subsection to cover the remaining case.

\begin{proof}[Proof of Lemma \ref{decene} for $\frac{N}{N-2} \le p \le \frac{N+2}{N-2}$]
As before, the proof is divided into 3 steps.

\medskip \noindent
\textsc{Step 1.} The claim that \eqref{eq_proof_22}-\eqref{eq_proof_21} continues to hold.
On the other hand, the comparison argument and \eqref{eq_dec} now imply that \eqref{eq_proof_11} holds for any pair $(\alpha, \nu)$ such that
\begin{equation}\label{eq_proof_32}
\alpha \ge 0 \quad \text{and} \quad 0 < \nu < \min\{N-2, (N-2)(q-1)-2+\alpha\}.
\end{equation}
The minimum in \eqref{eq_proof_22} is always non-negative (and positive if $p > \frac{N}{N-2}$), and that of \eqref{eq_proof_32} is always positive.

\medskip \noindent
\textsc{Step 2.} The behavior of the parameters $\alpha, \beta, \mu, \nu$ differs from the one in the previous subsection.
Because of this reason, in this time, for any $\eta>0$ there is a $C>0$ such that
\begin{equation}\label{eq_proof_31}
|\Psi(x)| \le \frac{C}{1+|x|^{N-2-\eta}}
\quad \text{and} \quad
|\Phi(x)| \le \frac{C}{1+|x|^{N-2-\eta}} \quad \text{in } \R^N.
\end{equation}
In fact, the iteration process produces a (finite) sequence $\{(\alpha_n, \beta_n)\}_{n \in \N}$ such that
\begin{itemize}
\item[-] Each of $\{\alpha_n\}$ and $\{\beta_n\}$ is increasing;
\item[-] For each $n$, it holds that $\beta_n \le \alpha_{n+1}$;
\item[-] $\alpha_{n+1}, \beta_n \nearrow N-2$ as $n$ gets large.
\end{itemize}

\medskip \noindent
\textsc{Step 3.} Having \eqref{eq_proof_31} in hand, one can conclude the proof of Lemma \ref{decene} by the same reasoning as before.
\end{proof}

%%%%%%%%%%%%%%%%%%%%%%%%%%%%%%%%%%%%%%%%%%%%%%%%%%%%%%%%%%%%%%%%%%%%%%%%%%%%%%%%
%%%%%%%%%%%

\bibliographystyle{amsalpha}

\end{document}